\documentclass[12pt]{article}
\usepackage{amsfonts}
\usepackage{amssymb}
\usepackage{mathrsfs}
\usepackage{srcltx}
\usepackage{amsmath}
\usepackage{amsthm}
\usepackage{amstext}
\usepackage{amsopn}
\usepackage{graphicx}

\newtheorem{theorem}{Theorem}[section]
\newtheorem{lemma}[theorem]{Lemma}
\newtheorem{proposition}[theorem]{Proposition}
\newtheorem{corollary}[theorem]{Corollary}

\theoremstyle{definition}

\theoremstyle{remark}
\newtheorem{remark}[theorem]{Remark}

\numberwithin{equation}{section}

\newcommand{\ba}{\begin{array}}
\newcommand{\ea}{\end{array}}
\newcommand{\f}{\frac}

\newcommand{\Om}{\Omega}
\newcommand{\cOm}{\overline{\Omega}}

\newcommand{\la}{\lambda}
\newcommand{\R}{{\mathbb R}}
\newcommand{\ds}{\displaystyle}

\begin{document}
\date{}
\title{ \bf\large{Dynamics of a diffusive predator-prey model: the effect of conversion rate}\footnote{This research is supported by the National Natural Science Foundation of China (Nos. 11371111 and 11301111)}}
 \author{Shanshan Chen\footnote{Email: chenss@hit.edu.cn},\ \  Junjie Wei\footnote{Corresponding Author, Email: weijj@hit.edu.cn},\ \ Jianhui Zhang
 \\
{\small  Department of Mathematics,  Harbin Institute of Technology,\hfill{\ }}\\
 {\small Weihai, Shandong, 264209, P.R.China.\hfill{\ }}}
\maketitle

\begin{abstract}
{A general diffusive predator-prey model is investigated in this paper.
We prove the global attractivity of constant equilibria when the conversion rate is small, and the
non-existence of non-constant positive steady states
when the conversion rate is large. The results are applied to several predator-prey models and give some ranges of parameters where complex pattern formation cannot occur.}
\noindent {\bf{Keywords}}: Reaction-diffusion; Non-existence; Steady state; Global attractivity
\end{abstract}

\section{Introduction}

During the past decades predator-prey interaction has been investigated extensively, and
there are several reaction-diffusion equations modelling the predator-prey interaction, see \cite{Du-Hsu,Du-Lou3,Du-Lou,Guo-Wu,Ko-Ryu,Leung,Yi} and references therein. The spatiotemporal patterns
induced by diffusion, such as Turing pattern, can be used to explain the complex phenomenon in ecology.
A prototypical one is the following diffusive predator-prey system with Holling type-II functional response
\begin{equation}\label{1.1}
\begin{cases}
  \ds\frac{\partial u}{\partial t}=d_1\Delta u+au\left(1-\ds\frac{u}{k}\right)-\ds\frac{buv}{1+mu}, & x\in \Omega,\; t>0,\\
 \ds\frac{\partial v}{\partial t}=d_2\Delta v-\theta v+\ds\frac{euv}{1+mu}, & x\in\Omega,\; t>0, \\
   \partial_\nu u=\partial_\nu v=0,& x\in \partial \Omega,\;
 t>0,\\
u(x,0)=u_0(x)\ge(\not\equiv)0, \;\; v(x,0)=v_0(x)\ge(\not\equiv)0,& x\in\Omega.
\end{cases}
\end{equation}
Yi, Wei and Shi \cite{Yi} investigated the Hopf and steady state bifurcations near the unique positive equilibrium of system \eqref{1.1}. Peng and Shi \cite{Peng-Shi} proved that the global bifurcating branches of steady state solutions are bounded loops containing at least two bifurcation
points, which improved the result in \cite{Yi}.
Ko and Ryu \cite{Ko-Ryu} investigated the dynamics of system \eqref{1.1} with a prey refuge. For the case of the homogeneous Dirichlet boundary condition, Zhou and Mu \cite{muzhou} showed the existence of positive steady states through bifurcation theory and fixed point index theory. Recently, Wang, Wei and Shi
\cite{Wangws} studied a diffusive predator-prey model in the following general form
\begin{equation}\label{1.1w}
\begin{cases}
  \ds\frac{\partial u}{\partial t}=d_1\Delta u+g(u)\left(f(u)-v\right), & x\in \Omega,\; t>0,\\
 \ds\frac{\partial v}{\partial t}=d_2\Delta v+v\left(-\theta +g(u)\right), & x\in\Omega,\; t>0,\\
   \partial_\nu u=\partial_\nu v=0,& x\in \partial \Omega,\;
 t>0,\\
u(x,0)=u_0(x)\ge(\not\equiv)0, \;\; v(x,0)=v_0(x)\ge(\not\equiv)0,& x\in\Omega,
\end{cases}
\end{equation}
where predator functional response $g(u)$ is increasing. They investigated the Hopf and steady state bifurcations near the unique positive equilibrium of system \eqref{1.1w}, and the existence and non-existence of non-constant positive steady states were also addressed with respect to the diffusion coefficients $d_1$ and $d_2$. Similar results on the Hopf and steady state bifurcations near the positive equilibrium can be found in \cite{Jin,Peng-Yi,Wangws2}.
Moreover, a non-monotonic functional response was proposed to model the prey's group defense, see \cite{Freed,Wolk}. That is, the predator functional response in model \eqref{1.1w} is non-monotonic and can be chosen as follows:
\begin{equation}\label{type4}
\text{(Holling type-IV)}\;\; g(u)=\ds\f{bu}{1+nu+m u^2},\;\;\text{where}\;\;b,n,m>0.
\end{equation}
Related to the work on model \eqref{1.1w} with non-monotonic functional response, see \cite{pangw,RuanX1,RuanX2,ZhuC} and references therein.
Another prototypical predator-prey model has the following from
\begin{equation}\label{DL}
\begin{cases}
  \ds\frac{\partial u}{\partial t}-d_1\Delta u=u(a-u)-\ds\frac{buv}{1+mu}, & x\in \Omega,\; t>0,\\
 \ds\frac{\partial v}{\partial t}-d_2\Delta v=v(d- v)+\ds\frac{euv}{1+mu}, & x\in\Omega,\; t>0,\\
    \partial_\nu u=\partial_\nu v=0,& x\in \partial \Omega,\;
 t>0,\\
u(x,0)=u_0(x)\ge(\not\equiv)0, \;\; v(x,0)=v_0(x)\ge(\not\equiv)0,& x\in\Omega.
\end{cases}
\end{equation}
Here the growth rate of the predator is logistic type in the absence of prey, and it is different from models \eqref{1.1} and \eqref{1.1w} that the predator can survive without the specific prey. For $m=0$, Leung \cite{Leung} proved the global attractivity of constant equilibria, which still holds when saturation $m$ is small \cite{Brown,Motoni,Du-Lou}.
Du and Lou \cite{Du-Lou} investigated the existence and non-existence of non-constant steady states when saturation $m$ is large, and see also \cite{Du-Lou3} for the case of the homogeneous Dirichlet boundary condition. Peng and Shi \cite{Peng-Shi} proved the non-existence of non-constant positive steady states.
 Moreover, Yang, Wu and Nie \cite{Yang-Wu} considered a diffusive predator-prey model under the homogeneous Dirichlet boundary condition, where the growth rate of the predator is like a Beverton-Holt function, and see \cite{Cheny} for the case of the homogeneous Neumman boundary condition.

Motivated by the above work of \cite{Du-Lou} and \cite{Wangws}, we analyze a diffusive predator-prey model in the following general form
\begin{equation}\label{1.4}
\begin{cases}
  \ds\frac{\partial u}{\partial t}=d_1\Delta u+g(u)\left(f(u)-v\right), & x\in \Omega,\; t>0,\\
 \ds\frac{\partial v}{\partial t}=d_2\Delta v+v\left(h(v) +cg(u)\right), & x\in\Omega,\; t>0, \\
   \partial_\nu u=\partial_\nu v=0,& x\in \partial \Omega,\;
 t>0,\\
u(x,0)=u_0(x)\ge(\not\equiv)0, \;\; v(x,0)=v_0(x)\ge(\not\equiv)0,& x\in\Omega.
\end{cases}
\end{equation}
Here $\Omega$ is a bounded domain in $\mathbb{R}^N$ ($N\le 3$) with a
smooth boundary $\partial \Omega$, $f(u)g(u)$ is the growth rate of prey without predator, $g(u)$ is the predator functional response, $c>0$ is the conversion rate, and $h(v)$ is the growth rate per capita of the predator in the absence of prey. For system \eqref{1.4}, we see that $(u,v)$ is a constant positive equilibrium if and only if $u\in(0,a)$ is a solution of the following equation
\begin{equation}
H(u):=h\left(f(u)\right)+cg(u)=0.
\end{equation}
However, it is hard to determine the exact numbers and expressions of constant positive equilibria of system \eqref{1.4}, not to mention the bifurcations near constant positive equilibria.
In this paper, we first assume that $f$ and $g$ satisfy the following assumptions:
\begin{enumerate}
\item[$(\mathbf{A_1})$] $f\in C^1(\overline{\mathbb {R}^+})$ and there exists a unique $a>0$ such that $f(u)$ is positive for $u\in[0,a)$ and negative for $u>a$. Moreover, $f'(u)<0$ on $(0,a]$ or there exists $\la\in(0,a)$ such that
    $f'(u)>0$ on $(0,\la)$ and $f'(u)<0$ on $(\la,a]$,
\item[$(\mathbf {A_2})$] $g\in C^1(\overline{\mathbb {R}^+})$, $g(0)=0$, and $g'(u)>0$ for $u\ge0$,
\end{enumerate}
and investigate the effect of small conversion rate on the positive steady states of system \eqref{1.4}.
There are several examples of $f$ and $g$ satisfying assumptions $(\mathbf{A_1})$ and $(\mathbf{A_2})$. For example,
\begin{enumerate}
\item [(1)] Richards growth rate for prey and Holling type-II functional response:
\begin{equation*}
  f(u)=\ds\f{\gamma(1+mu)(a-u^p)}{b}\;\;\text{and}\;\;g(u)=\ds\f{bu}{1+mu},
\end{equation*}
where $a,b,m,\gamma>0$ and $p\ge1$;
\item [(2)] weak Allee effect in prey and Holling type-II functional response:
\begin{equation*}
  f(u)=\ds\f{\gamma(1+mu)(a-u)(u+p)}{b}\;\;\text{and}\;\;g(u)=\ds\f{bu}{1+mu},
  \end{equation*}
  where $a,b,p,m,\gamma>0$ and $a>p$;
\item [(3)] logistic growth rate for prey and Ivlve type functional response:
\begin{equation*}
 f(u)=\begin{cases}
   \ds\f{\gamma u(a-u)}{\alpha\left(1-e^{-\beta u}\right)}, & \text{ for } u\ne 0,\\
  \ds\f{a\gamma}{\alpha\beta}, & \text{ for }u=0,
\end{cases}\;\;\text{and}\;\;g(u)=\alpha\left(1-e^{-\beta u}\right),
\end{equation*}
where $a,\alpha,\beta,\gamma>0$.
\end{enumerate}
Then, we consider the case that $f$ and $g$ satisfy the following assumptions:
\begin{enumerate}
\item[$(\mathbf{A'_1})$] $f\in C^1(\overline{\mathbb {R}^+})$ and there exists a unique $a>0$ such that $f(u)$ is positive for $u\in[0,a)$ and negative for $u>a$,
\item[$(\mathbf{A'_2})$] $g\in C^2(\overline{\mathbb {R}^+})$, $g(0)=0$, $g'(0)>0$ and $g(u)>0$ for $u>0$,
\end{enumerate}
and investigate the effect of large conversion rate on the positive steady states of system \eqref{1.4}. Here we remark that $(\mathbf{A'_1})$ is more general that $(\mathbf{A_1})$, $g$ may be nonmonotonic for $(\mathbf{A'_2})$, and the assumption that $g\in C^2$ is needed to guarantee the regularity of the positive steady states. There are also several examples of $f$ and $g$ satisfying assumptions $(\mathbf{A'_1})$ and $(\mathbf{A'_2})$. For example, \begin{enumerate}
\item [(1)] logistic growth rate for prey and Holling type-IV functional response:
\begin{equation*}
  f(u)=\ds\f{\gamma(1+nu+mu^2)(a-u)}{b}\;\;\text{and}\;\;g(u)=\ds\f{bu}{1+nu+m u^2},
\end{equation*}
where $a,b,m,n,\gamma>0$.
\end{enumerate}
The rest of the paper is organized as follows.
In Section 2, we show the global attractivity of constant equilibria of system \eqref{1.4} when the conversion rate is small, which also implies the non-existence of non-constant positive steady states.
In Section 3, we prove the non-existence of non-constant positive steady states of system \eqref{1.4} when the conversion rate is large. In Section 4, we apply the obtained theoretical results to some concrete examples.

\section {The case of small conversion rate}
 In this section, we investigate the positive steady states of system \eqref{1.4} when conversion rate $c$ is small.
 Throughout this section, we assume that $h$ satisfies the following assumption
\begin{enumerate}
\item[$(\mathbf{A_3})$] $h\in C^1(\overline{\mathbb {R}^+})$, and there exists a unique $d>0$ such that $h(v)$ is positive for $v\in[0,d)$, and negative for $v>d$. Moreover, $h'(v)< 0$ for $v\ge d$.
\end{enumerate}
Then $h(v)$ has a inverse function, denoted by $h^{-1}$, when $v\in[d,\infty)$.
 We first recall the following well-known result for later application.
 \begin{lemma}\label{L2.0}
 Assume that $H:(0,\infty)\to \R$ is a smooth function satisfying $H(w)(w-w_0)<0$ for any $w>0$ and $w\ne w_0$.
If $w(x,t)$ satisfies the following problem
\begin{equation*}\label{aux}
\begin{cases}
\displaystyle \frac{\partial w}{\partial t}=d\Delta w+H(w),  & \;\; x\in\Omega,\; t>t_0,\\
\displaystyle \frac{\partial w(x,t)}{\partial \nu}=0, & \;\; x\in \partial\Omega,\; t>t_0,\\
\displaystyle  w(x,t_0)\ge(\not\equiv)0, & \;\; x\in \Om,
\end{cases}
\end{equation*}
where $d>0$, $t_0\in \R^+$, then $w(x,t)$ exists for all $t>t_0$, and $w(x,t)\rightarrow w_0$ uniformly for $x\in \overline
\Omega$ as $t\rightarrow \infty$.
\end{lemma}
From Lemma \ref{L2.0}, we give the exact asymptotic bounds of the solutions for system \eqref{1.4}.
\begin{lemma}\label{L2.1}
Assume that $f$, $g$ and $h$ satisfy assumptions $(\mathbf{A_1})$, $(\mathbf{A_2})$ and $(\mathbf{A_3})$. If   $h(f(0))<-cg(a)$, then
there exist $(\underline u,\underline v),\;(\overline
u,\overline v)>(0,0)$ satisfy
\begin{equation}\label{re}
 \begin{split}
       &f(\overline u)-\underline v\le0 ,\;\;h(\overline v)+cg(\overline u)\le0, \\
      &f(\underline u)-\overline v\ge0 ,\;\;h(\underline v)+cg(\underline u)\ge0. \\
    \end{split}
\end{equation}
Moreover, for any initial value $\phi=(u_0(x),v_0(x))$, where $u_0(x)\ge(\not\equiv)0$,
$v_0(x)\ge(\not\equiv)0$ for all $x\in \cOm$, there exists $t_0(\phi)>0$ such that
the corresponding solution $(u(x,t),v(x,t))$ of system \eqref{1.4}
satisfies
\begin{equation}\label{re2}
(\underline u,\underline v)\le(u(x,t),v(x,t))\le(\overline
u,\overline v)
\end{equation}
for any $t>t_0(\phi)$.
\end{lemma}
\begin{proof}
Since $h(f(0))<-cg(a)$, if follows from assumption $(\mathbf{A_3})$ that there exists $\epsilon>0$ such that $h^{-1}\left(-cg(a+\epsilon)\right)>d$ exists and
\begin{equation}\label{fg1}
\epsilon<d,\;\;f(0)-\left[h^{-1}\left(-cg(a+\epsilon)\right)+\epsilon\right]>0,
\end{equation}
where $h^{-1}$ is the inverse function of $h$.
It follows from $$ \ds\frac{\partial u}{\partial t}\le d_1\Delta u+g(u)f(u)$$
and Lemma \ref{L2.0} that,
for any initial value $\phi$, there exists $t_1(\phi)>0$ such that $u(x,t)\le  a+\epsilon$ for $t>t_1(\phi)$, where $a>0$ is the unique zero of $f(u)$.
Since $$\ds\frac{\partial v}{\partial t}\ge d_2\Delta v+vh(v),$$ there exists
$t_2(\phi)>t_1(\phi)$ such that $v(x,t)\ge  d-\epsilon>0$ for $t>t_2(\phi)$.
Consequently, we have
$$\ds\frac{\partial v}{\partial t}\le d_2\Delta v+v\left[h(v)+cg(a+\epsilon)\right]\;\;\text{for}\;\;t>t_2(\phi),$$
and then there exists $t_3(\phi)>t_2(\phi)$ such that $v(x,t)\le h^{-1}\left(-cg(a+\epsilon)\right)+\epsilon$ for $t>t_3(\phi)$.
Since $f,g$ satisfy $(\mathbf{A_1})$, $(\mathbf{A_2})$ and Eq. \eqref{fg1}, there exists $\overline a \in(0,a)$, depending on $c$, such that
$$f(u)-\left[h^{-1}\left(-cg(a+\epsilon)\right)+\epsilon\right]>0 \text{ for }u\in[0,\overline a],$$
and $$f(u)-\left[h^{-1}\left(-cg(a+\epsilon)\right)+\epsilon\right] <0\text{ for }u\in(\overline a, a].$$
It follows from Lemma \ref{L2.0} that there exists $t_4(\phi)>t_3(\phi)$ such that
$u(x,t)\ge \ds\f{\overline a}{2}>0$ for $t>t_4(\phi)$.
Choose
\begin{equation}\label{re3}
 \begin{split}
       &\overline u=a+\epsilon ,\;\;\underline u=\ds\f{\overline a}{2}, \\
      &\underline v= d-\epsilon,\;\;\overline v=h^{-1}\left(-cg(a+\epsilon)\right)+\epsilon. \\
    \end{split}
\end{equation}
Then $(\underline u,\underline v)$ and $(\overline
u,\overline v)$ satisfy Eq. \eqref{re}, and
there exists $t_0(\phi)>t_4(\phi)$ such that
the corresponding solution $(u(x,t),v(x,t))$ of system \eqref{1.4}
satisfies Eq. \eqref{re2}
for any $t>t_0(\phi)$.
\end{proof}
Then, through the upper and lower solution method
\cite{27,pao,4,5}, we have the following results on the global attractivity of the positive equilirium.
\begin{theorem}\label{T2.2}
Assume that $f$, $g$ and $h$ satisfy assumptions $(\mathbf{A_1})$, $(\mathbf{A_2})$ and $(\mathbf{A_3})$. If $f(0)>d$, then there exists $c_0>0$, depending on $f$, $g$ and $h$, such that system \eqref{1.4} has a unique constant positive
steady state solution $(u_*,v_*)$ for any $c\in (0,c_0)$. Furthermore,
for any initial value $\phi=(u_0(x),v_0(x))$, where $u_0(x)\ge(\not\equiv)0$,
$v_0(x)\ge(\not\equiv)0$ for all $x\in \cOm$, the corresponding solution $(u(x,t),v(x,t))$ of system
\eqref{1.4} converges uniformly to $(u_*,v_*)$ as
$t\rightarrow\infty$.
\end{theorem}
\begin{proof}
Since $f(0)>d$ (or equivalently, $h(f(0))<0$), there exists $c_1>0$ such that $h(f(0))<-cg(a)$ for $c\in(0, c_1]$. In the following, we always assume that
$c<c_1$.
From Lemma \ref{L2.1}, we see that there exist $(\underline u,\underline v),\;(\overline
u,\overline v)>(0,0)$, which is a pair of coupled upper and lower solution of system
\eqref{1.4}. Since $f$ and $g$ are smooth, there exists $K>0$ such that, for any $(u_1,v_1)$ and $(u_2,v_2)$
satisfying $(\underline u,\underline v)\le(u_1,v_1),(u_2,v_2)\le(\overline
u,\overline v)$, we have
\begin{equation*}
\begin{array}{l}
\left|g(u_1)\left(f(u_1)-v_1\right)-g(u_2)\left(f(u_2)-v_2\right)\right|\le K(|u_1-u_2|+|v_1-v_2|),\\
\left|v_1h(v_1)+cg(u_1)v_1-v_2h(v_2)-cg(u_2)v_2\right|\le
K(|u_1-u_2|+|v_1-v_2|).
\end{array}
\end{equation*}
Define two iteration sequences $(\overline u^{(m)},\overline
v^{(m)})$ and $(\underline u^{(m)},\underline v^{(m)})$ as follows: for
$m\ge 0$,
\begin{equation*}
\begin{array}{l}
\overline u^{(m+1)}=\overline
u^{(m)}+\ds\frac{g\left(\overline u^{(m)}\right)}{K}\left[f(\overline u^{(m)})-\underline v^{(m)}\right],\\
\underline u^{(m+1)}=\underline
u^{(m)}+\ds\frac{g\left(\underline u^{(m)}\right)}{K}\left[f(\underline u^{(m)})-\overline v^{(m)}\right],\\
\overline v^{(m+1)}=\overline
v^{(m)}+\ds\frac{\overline v^{(m)}}{K}\left[h\left(\overline v^{(m)}\right)+cg\left(\overline u^{(m)}\right)\right],\\
\underline v^{(m+1)}=\underline
v^{(m)}+\ds\frac{\underline v^{(m)}}{K}\left[h\left(\underline  v^{(m)}\right)+cg\left(\underline  u^{(m)}\right)\right],\\
\end{array}
\end{equation*}
where $(\overline u^{(0)}, \overline v^{(0)})=(\overline
u,\overline v)$ and $(\underline u^{(0)}, \underline v^{(0)})=(\underline
u,\underline v)$. Then there exist $(\tilde u,\tilde v)$ and
$(\check u, \check v)$ such that $(\underline u,\underline v)\le(\check u, \check v)\le (\tilde u,\tilde v)\le(\overline
u,\overline v)$, $\lim_{m
\rightarrow \infty}\overline u^{(m)}=\tilde u$, $\lim_{m \rightarrow
\infty}\overline v^{(m)}=\tilde v$, $\lim_{m \rightarrow
\infty}\underline u^{(m)}=\check u$, $\lim_{m \rightarrow
\infty}\underline v^{(m)}=\check v$, and
\begin{equation}\label{2.5}
    \begin{split}
       &0=f(\tilde u)-\check v,~0=h(\tilde v)+cg(\tilde u), \\
       &0= f(\check u)-\tilde v,~0=h(\check v)+cg(\check u).
    \end{split}
\end{equation}
It follows from Eq. \eqref{2.5} that
\begin{equation}\label{fgqu}
h\left(f(\tilde u)\right)-c g(\tilde u)=h\left(f(\check u)\right)-c g(\check u).
\end{equation}
The following proof is given in two cases.\\
Case I. If $f'(u)<0$ on $(0,a]$, where $a$ is the unique zero of $f(u)$, then it follows from Eq. \eqref{2.5} and assumptions $(\mathbf{A_1})$-$(\mathbf{A_3})$ that
\begin{equation}\label{inc}
f(\check u)=\tilde v=h^{-1}\left(-cg(\tilde u)\right)<h^{-1}\left(-c_1g(a)\right)<f(0)\;\;\text{for any}\;\;c\in(0,c_1).\end{equation}
Then there exists $\overline \lambda\in(0,a)$ such that $f(\overline \lambda)=h^{-1}\left(-c_1g(a)\right)$ and $\check u\ge \overline \la $ for any $c\in (0,c_1)$. Hence
$$0=h(f(\tilde u))-c g(\tilde u)-h(f(\check u))+c g(\check u)\ge \left[\min_{u\in[\overline \la,a]}\left[h\left(f(u)\right)\right]'-c\max_{u\in[\overline \la,a]}g'(u)\right](\tilde u-\check u) $$
for any $c\in(0,c_1)$.
Since $\min_{u\in[\overline \la,a]}\left[h\left(f(u)\right)\right]'>0$, there exists $c_0\in(0,c_1)$, depending on $f$, $g$ and $h$, such that $\tilde u=\check u$ for any $c\in(0,c_0)$.\\
Case II. If there exists $\la\in(0,a)$ such that
    $f'(u)>0$ on $(0,\la)$ and $f'(u)<0$ on $(\la,a]$, then there exist $\tilde \lambda\in(\la,a)$ such that $f(\tilde \la)=f(0)$. It follows from Eq. \eqref{inc} that $\check u \ge \tilde \la$ for any $c\in(0,c_1)$. Hence
$$0=h(f(\tilde u))-c g(\tilde u)-h(f(\check u))+c g(\check u)\ge \left[\min_{u\in[\tilde \la,a]}\left[h\left(f(u)\right)\right]'-c\max_{u\in[\tilde \la,a]}g'(u)\right](\tilde u-\check u) $$
for any $c\in(0,c_1)$.
Since $\min_{u\in[\tilde \la,a]}\left[h\left(f(u)\right)\right]'>0$, there exists $c_0\in(0,c_1)$, depending on $f$, $g$ and $h$, such that $\tilde u=\check u$ for $0<c<c_0$.

It follows from the upper and lower solution method
\cite{27,pao,4,5} that, for any $c\in(0,c_0)$, system \eqref{1.4} has a unique constant positive equilibrium $(u_*,v_*)$, which is globally attractive.
\end{proof}
From Lemma \ref{L2.0} and Theorem \ref{T2.2}, we have the following results on the global attractivity of constant equilibria.
\begin{corollary}\label{cc25}
Assume that $f$, $g$ and $h$ satisfy assumptions $(\mathbf{A_1})$, $(\mathbf{A_2})$ and $(\mathbf{A_3})$. Then the following two statements are true.
\begin{enumerate}
\item [$(1)$] If $f(0)>d$ (or equivalently, $h(f(0))<0$), then there exists $c_0>0$, depending on $f$, $g$ and $h$, such that, for any $c\in(0,c_0)$, system \eqref{1.4} has a unique constant positive steady state, which is globally attractive. Hence system \eqref{1.4} has no non-constant positive steady states for $c\in(0,c_0)$ if $f(0)>d$.
\item [$(2)$] If $\max_{u\in[0,a]}f(u)<d$ (or equivalently $h(\max_{u\in[0,a]}f(u))>0$), then the steady state $(0,d)$ of system \eqref{1.4} is globally attractive. Hence system \eqref{1.4} has no positive steady states if $\max_{u\in[0,a]}f(u)<d$.
\end{enumerate}
\end{corollary}
\begin{proof}
If $f(0)>d$, then it follows from Theorem \ref{T2.2} that there exists $c_0>0$ such that, for $c\in(0,c_0)$, system \eqref{1.4} has a unique constant positive steady state, which is globally attractive. Hence system \eqref{1.4} has no non-constant positive steady states for $c\in(0,c_0)$.

Then, we consider the case that $\max_{u\in[0,a]}f(u)<d$. Since $$\ds\frac{\partial v}{\partial t}\ge d_2\Delta v+vh(v),$$ it follows from Lemma \ref{L2.0} that
for any initial value and $$\epsilon\in\left(0,\ds\f{1}{2}\left(d-\max_{u\in[0,a]}f(u)\right)\right),$$ there exists
$t_2>0$ such that $v(x,t)\ge  d-\epsilon>0$ for $t>t_2$, and consequently,
$$ \ds\frac{\partial u}{\partial t}\le d_1\Delta u+g(u)\left(f(u)-v\right)\le d_1\Delta u+g(u)\left(\max_{u\in[0,a]}f(u)-d+\epsilon\right)$$for $t>t_2$, which implies that $u(x,t)$ converges uniformly to $0$ as $t\to\infty$. Therefore, the steady state $(0,d)$ of system \eqref{1.4} is globally attractive.

\end{proof}
We remark that the above results do not mention the case that $$ h\left(\max_{u\in[0,a]}f(u)\right)<0<h(f(0)),$$ that is,
$$\max_{u\in[0,a]}f(u)>d>f(0).$$ For this case, the dynamics is complex even when $c=0$. The steady states of Eq. \eqref{1.4} for $c=0$ satisfy
\begin{equation}\label{steady1.4}
\begin{cases}
  -d_1\Delta u=g(u)\left(f(u)-v\right), & x\in \Omega,\\
 -d_2\Delta v=vh(v), & x\in\Omega, \\
   \partial_\nu u=\partial_\nu v=0,& x\in \partial \Omega,\\
\end{cases}
\end{equation}
and $(u(x),d)$ is a positive solution of Eq. \eqref{steady1.4} if and only if $u(x)$ is a positive solution of the following equation
\begin{equation}\label{allee}
\begin{cases}
-d_1\Delta u=g(u)\left(f(u)-d\right),& x\in\Omega, \\
   \partial_\nu u=0,& x\in \partial \Omega.\\
   \end{cases}
\end{equation}
Eq. \eqref{allee} may exhibit a strong Allee effect, which has a stable constant positive equilibrium and an unstable constant positive equilibrium, see \cite{dsams,Wangws2,wangsjmb} for related work on strong Allee effect.

\section{The case of large conversion rate}
In this section, we prove the non-existence of non-constant positive steady states of system \eqref{1.4} when conversion rate $c$ is large, and
the method used here is motivated by \cite{Peng-Shi}. Throughout this section, we assume that $f$ and $g$ satisfy $(\mathbf{A'_1})$ and $(\mathbf {A'_2})$.
Define \begin{equation}\label{pu}
q(u)=\begin{cases}
   \ds\f{g(u)}{u}, & \text{ if } u>0,\\
  g'(0), & \text{ if }u=0,
\end{cases}
\end{equation}
and then $q(u)\in C^1(\overline{\mathbb {R}^+})$.
Let $w=cu$, $\rho=\ds\f{1}{c}$, and then $w$ and $v$ satisfy
\begin{equation}\label{steadywz1.4}
\begin{cases}
  -d_1\Delta w=wq(\rho w)\left[f(\rho w)-v\right], & x\in \Omega,\\
 -d_2\Delta v=v\left[h(v) +q(\rho w)w\right], & x\in\Omega, \\
   \partial_\nu u=\partial_\nu v=0,& x\in \partial \Omega.\\
\end{cases}
\end{equation}
Hence the existence/non-existence of positive steady states of system \eqref{1.4} for large $c$ is equivalent to that of solutions of system \eqref{steadywz1.4} for small $\rho$.
In the following, we first cite the maximum principle for weak solutions from \cite{Lieberman,Lou-Ni,Peng-Shi-Wang2} and the Harnack inequality for weak solutions from \cite{Lin-Ni,Peng-Shi,Peng-Shi-Wang2} for later application.
\begin{lemma}\label{maxin}
Assume that $\Omega$ is a bounded Lipschitz domain in $\mathbb{R}^N$, and $g\in C(\overline\Omega\times \mathbb R)$. If $z\in H^{1}(\Omega)$
is a weak solution of the inequalities
\begin{equation*}
\begin{cases}
  \Delta z+g(x,z)\ge0, & x\in \Omega,\\
 \partial_\nu  z\le0,& x\in \partial \Omega,\\
\end{cases}
\end{equation*}
and there exists a constant $K$ such that $g(x, z) < 0 $ for $z > K$, then
$$z \le K \text{ a.e. in } \Omega.$$
\end{lemma}
\begin{lemma}\label{harnack}
Assume that $\Omega$ is a bounded Lipschitz domain in $\mathbb{R}^N$, and $c(x)\in L^q(\Omega)$ for some $q > N/2$. If
$z \in H^{1}(\Omega)$ is a non-negative weak solution of the following problem
\begin{equation*}
\begin{cases}
  \Delta u +c(x) u=0, & x\in \Omega,\\
 \partial_\nu  u=0,& x\in \partial \Omega,\\
\end{cases}
\end{equation*}
then there is a positive constant $C$, which is determined only by $\|c(x)\|_q$, $q$ and $\Omega$, such that
$$\sup_{x\in\Omega} u\le C\inf_{x\in\Omega} u.$$
\end{lemma}
Then, we consider the positive solutions of system \eqref{steadywz1.4} when $\rho=0$.
\begin{lemma}\label{rho0}
Assume that $f$ and $g$ satisfy assumptions $(\mathbf{A'_1})$ and $(\mathbf{A'_2})$, $h$ satisfies
\begin{enumerate}
\item[$(\mathbf{A_4})$] $\left[h(v)-h\left(f(0)\right)\right]\left(v-f(0)\right)<0$ for any $v>0(v\ne f(0))$,
\end{enumerate}
and $h\left(f(0)\right)<0$.
Then, for $\rho=0$, system \eqref{steadywz1.4} has a unique positive steady state $\left(-\f{h(f(0))}{g'(0)},f(0)\right)$.
\end{lemma}
\begin{proof}
Since $h(f(0))<0$, system \eqref{steadywz1.4} has a constant positive steady state $(w_*,v_*)=\left(-\f{h(f(0))}{g'(0)},f(0)\right)$ for $\rho=0$.
We construct the following function
\begin{eqnarray*}
G(w,v)&:=&\int_{\Omega}\left\{\ds\f{w-w_*}{w}\left[d_1\Delta w+wg'(0)(f(0)- v)\right]\right \}dx\\
&+&\int_{\Omega}\left\{\ds\f{v-v_*}{v}\left[d_2\Delta v+v\left(h(v)+ g'(0)w\right)\right]\right\}dx\\
&=&-\int_{\Omega}\left[d_1\ds\f{w_*|\nabla w|^2}{w^2}+d_2\ds\f{v_*|\nabla v|^2}{v^2}\right]dx\\
&+&\int_{\Omega}(v-v_*)[h(v)-h(v_*)]dx.
\end{eqnarray*}
Therefore, if $(w(x),v(x))$ is a positive solution of system \eqref{steadywz1.4} for $\rho=0$, then $$G(w(x),v(x))=0.$$
Since $h$ satisfies $(\mathbf{A_4})$, it follows that $(w(x),v(x))\equiv(w_*,v_*)$.
\end{proof}

Based on Lemmas \ref{maxin} and \ref{harnack}, we will give a \textit{priori} estimates for positive solutions of system \eqref{steadywz1.4} under the following assumption $(\mathbf{A_5})$.
\begin{enumerate}
\item [$(\mathbf{A_5})$]  $h\in C^1(\overline{\mathbb {R}^+})$ and there exist $n\in \mathbb{N}^+$, $\{q_i\}_{i=0}^{n}$, $\{k_i\}_{i=0}^n$ and $\{\overline k_i\}_{i=0}^n$ such that $$\sum_{i=0}^n k_i v^{q_i}\le -h(v)\le \sum_{i=0}^n \overline k_i v^{q_i}\;\;\text{for any}\;\;v\ge0,$$
    where $0= q_0<q_1<q_2<\dots< q_n$, $q_n>\ds\f{1}{2}$ and $k_n,\overline k_n>0$.
\end{enumerate}
The above mentioned assumptions $(\mathbf{A_4})$ and $(\mathbf{A_5})$ are not strong, and we will remark that many common used growth rate per capita functions satisfy $(\mathbf{A_4})$ and $(\mathbf{A_5})$ at the end of this section.
\begin{theorem}\label{imesti}
Assume that $f$, $g$ and $h$ satisfy assumptions $(\mathbf{A'_1})$, $(\mathbf{A'_2})$ and $(\mathbf{A_5})$, $\Omega$ is a bounded domain in $\mathbb{R}^N(N\le3)$ with a smooth boundary $\partial\Omega$, and $(w_{\rho},v_{\rho})$ is a positive solution of system \eqref{steadywz1.4}. Denote
 \begin{equation}\label{pu2}
v_0=\begin{cases}
   \max\{v\ge0:h(v)=0\}, & \text{ if } \{v\ge0:h(v)=0\}\ne\emptyset,\\
  0, & \text{ if }\{v\ge0:h(v)=0\}=\emptyset.
\end{cases}
\end{equation}
Then the following two statements are true.
\begin{enumerate}
\item [$(1)$] There exists $\overline C>0$ such that $\sup_{x\in\Omega} w_{\rho},\sup_{x\in\Omega}v_{\rho}<\overline C$ for all $\rho>0$.
\item [$(2)$] If $f(0)>v_0$, then there exists $M>0$ such that
$$\inf_{0\le \rho\le M}\inf_{x\in \Omega}w_{\rho}>0\;\;\text{and}\;\; \inf_{0\le \rho\le M}\inf_{x\in\Omega}v_{\rho}>0.$$
\end{enumerate}
\end{theorem}
\begin{proof}
Since $h$ satisfies assumption $(\mathbf{A_5})$, we see that $\lim_{v\to\infty}h(v)=-\infty$, which implies that $v_0$ is well defined.
From Eq. \eqref{steadywz1.4}, we obtain
\begin{equation}\label{eswv}
\begin{split}
&-\int_{\Omega}v_{\rho}h(v_\rho)dx=\int_{\Omega}w_{\rho}q\left(\rho w_{\rho}\right)f\left(\rho w_{\rho}\right)dx,\\
&\int_{\Omega}\left[h(v_\rho)+w_\rho q(\rho w_\rho)\right]dx=-d_2\int_{\Omega}\ds\f{|\nabla v_{\rho}|^2}{v_{\rho}^2}dx\le0.
\end{split}
\end{equation}
Since $\rho w_\rho$ satisfies $$-d_1\rho\Delta  w_{\rho}\le\rho w_{\rho} f(\rho w_\rho)g(\rho w_\rho),$$
it follows from Lemma \ref{maxin} (see also \cite{Lieberman,Lou-Ni,Peng-Shi-Wang2}) that $\rho w_\rho\le a$, where $a>0$ is the unique zero of $f(u)$.
Noticing that $$\max_{u\in[0,a]}f(u),\max_{u\in[0,a]}q(u),\min_{u\in[0,a]}q(u)>0,$$
from Eq. \eqref{eswv}, we have
\begin{equation}\label{hv}
-\int_{\Omega}v_{\rho}h(v_\rho)dx\le -\ds\f{\max_{u\in[0,a]}f(u)\max_{u\in[0,a]}q(u)}{\min_{u\in[0,a]}q(u)}\int_{\Omega}h(v_\rho)dx.
\end{equation}
This relation and assumption $(\mathbf{A_5})$ imply
\begin{equation}\label{isum}
\begin{split}
-\int_{\Omega}h(v_\rho)dx\le& \sum_{i=0}^{n}\overline \sigma_i\|v_\rho\|^{q_i}_{q_n+1},\\
k_n \|v_\rho\|_{q_n+1}^{q_n+1}\le& \sum_{i=0}^{n-1}\sigma_i\|v_\rho\|^{q_i+1}_{q_n+1}+\sum_{i=0}^{n}\overline \sigma_i\|v_\rho\|^{q_i}_{q_n+1},
\end{split}
\end{equation}
where $\{\sigma_i\}_{i=0}^{n-1}$ and $\{\overline\sigma_i\}_{i=0}^n$ are positive and depend only on $f$, $g$, $\{q_i\}_{i=0}^n$, $\{k_i\}_{i=0}^n$, $\{\overline k_i\}_{i=0}^n$ and $\Omega$.
Therefore, there exists a constant $C_1>0$ such that
$\|v_\rho\|_{q_n+1}\le C_1$ for all $\rho\ge0$,
 which implies
$$\|q(\rho w)\left[f(\rho w)-v\right]\|_{q_n+1}\le\max_{u\in[0,a]}q(u)\left(\max_{u\in[0,a]}f(u)|\Omega|^{\f{1}{q_n+1}}+C_1\right).$$
Since $N\le 3$ and $q_n+1>\ds\f{3}{2}\ge\ds\f{N}{2}$, from Lemma \ref{harnack} we see that there exists $C_2>0$ such that
\begin{equation}\label{wharnack}\sup_{x\in\Omega} w_\rho \le C_2\inf_{x\in\Omega}w_\rho \;\;\text{for all}\;\;\rho\ge0.\end{equation}
It follows from Eqs. \eqref{eswv} and \eqref{isum} that
\begin{equation}\label{infw}
\inf_{x\in\Omega}w_{\rho}<\ds\f{\sum_{i=0}^{n}\overline \sigma_iC_1^{q_i}}{\min_{u\in[0,a]}q(u)|\Omega|}\;\;\text{for all}\;\;\rho\ge0.\end{equation}
Therefore, from Eq. \eqref{wharnack} and \eqref{infw}, we see that there exists a constant $C_3>0$ such that
\begin{equation}\label{wupper}
\sup_{x\in\Omega}w_\rho\le C_3\;\;\text{for all}\;\;\rho\ge0.
\end{equation}
 Consequently,
\begin{equation}\label{vupper}
-d_2\Delta v_\rho=v_\rho\left[h(v_\rho) +q(\rho w_\rho)w_\rho\right]\le v_\rho\left[h(v_\rho)+\max_{u\in[0,a]}q(u)C_3\right].
\end{equation}
It follows from assumption $(\mathbf{A_5})$ that $\lim_{v\to\infty}h(v)=-\infty$, which implies that there exists $C_4>0$
such that
\begin{equation}\label{C4}
h(v)+\max_{u\in[0,a]}q(u)C_3<0\;\;\text{for}\;\; v>C_4.
\end{equation}
Therefore, from Eqs. \eqref{vupper} and \eqref{C4} and Lemma \ref{maxin}, we obtain
\begin{equation}
\sup_{x\in\Omega} z_\rho\le C_4\;\;\text{for all}\;\;\rho\ge0.
\end{equation}
Letting $\overline C=\max\{C_3,C_4\}$, we have
\begin{equation}\label{esup}\sup_{x\in\Omega} w_{\rho},\;\sup_{x\in\Omega}v_{\rho}<\overline C\;\;\text{for all}\;\;\rho\ge0.\end{equation}

In the following, we find the lower bound of $w_\rho$ and $v_\rho$.
We first claim that there exists $M_1>0$ such that
\begin{equation}\label{inzz}\inf_{0\le \rho\le M_1}\inf_{x\in \Omega}w_{\rho}>0.\end{equation}
By way of contradiction, there exists a sequence $\{\rho_{j}\}_{j=1}^\infty$ such that $\lim_{j\to\infty}\rho_{j}=0$ and
$\lim_{j\to\infty}\inf_{x\in\Omega} w_{\rho_{j}}=0$, which implies that $\lim_{j\to\infty}\sup_{x\in\Omega} w_{\rho_{j}}=0$ from Eq. \eqref{wharnack}.
Then we only need to consider two cases.\\
Case I. $\{v\ge0:h(v)=0\}=\emptyset$. Then $v_0=0<f(0)$.
Noticing that $\lim_{v\to\infty}h(v)=-\infty$ from assumption $(A_5)$, we obtain $h(v)<0$ for $v\ge0$ and $\max_{v\ge0}h(v)<0$. Since $\lim_{j\to\infty}\sup_{x\in\Omega} w_{\rho_{j}}=0$, we have
$$\int_{\Omega} v_{\rho_{j}}\left[-h(v_{\rho_{j}}) -q(\rho_j w_{\rho_{j}})w_{\rho_{j}}\right]dx\ge\int_\Omega v_{\rho_{j}}\left[-\max_{v\ge0}h(v)-q(\rho_j w_{\rho_{j}})w_{\rho_{j}}\right]>0$$
for sufficiently large $j$, which contradicts with the fact that
$$\int_{\Omega} v_{\rho_{j}}\left[-h(v_{\rho_{j}}) -q(\rho_j w_{\rho_{j}})w_{\rho_{j}}\right]dx=0.$$
Case II. $\{v\ge0:h(v)=0\}\ne \emptyset$ and $v_0<f(0)$.
Since $\lim_{j\to\infty}\sup_{x\in\Omega} w_{\rho_{j}}=0$,
we see that, for any $\epsilon>0$, there exists $j_0(\epsilon)>0$ such that
$\sup_{x\in\Omega}|w_{\rho_{j}}|\max_{u\in[0,a]} q(u)<\epsilon$ for any $j>j_0(\epsilon)$, which implies that
\begin{equation*}-d\Delta v_{\rho_{j}}\le v_{\rho_{j}}\left[h\left(v_{\rho_{j}}\right)+\epsilon\right]\;\;\text{for}\;\;j>j_0(\epsilon).\end{equation*}
This relation and Lemma \ref{maxin} lead to $$v_{\rho_j}\le v_\epsilon:=\max\{v\ge0:h(v)+\epsilon=0\}\;\;\text{for}\;\;j>j_0(\epsilon).$$
It follows from $\lim_{\epsilon\to0}v_\epsilon=v_0<f(0)$ that $v_\epsilon <f(0)$ for sufficiently small $\epsilon$, and without loss of generality, we assume $v_\epsilon <f(0)$. Since
\begin{equation}\label{ppw}
\begin{split}
-d_1\rho_{j} \Delta w_{\rho_{j}}=&\rho_{j} w_{\rho_{j}} g\left(\rho_{j} w_{\rho_{j}}\right)\left[f\left( \rho_{j}w_{\rho_{j}}\right)-v_{\rho_{j}}\right]\\ \ge& \rho_{j} w_{\rho_{j}} g\left(\rho_{j} w_{\rho_{j}}\right)\left[f\left( \rho_{j}w_{\rho_{j}}\right)-v_\epsilon\right]
\end{split}
\end{equation}
for $j>j_0(\epsilon)$, it follows from Lemma \ref{maxin} and $v_\epsilon <f(0)$ that there exists $w_0>0$ such that
$\rho_{j}w_{\rho_{j}}\ge w_0>0$ for $j>j_0(\epsilon)$, which contradicts with the fact that
$\lim_{j\to\infty}\rho_{j}=0$ and $\lim_{j\to\infty}\sup_{x\in\Omega} w_{\rho_{j}}=0$.\\
Therefore, the claim is proved and Eq. \eqref{inzz} hold. Then, we prove that there exists $M_2>0$ such that
\begin{equation}\label{inzz2}\inf_{0\le \rho\le M_2}\inf_{x\in \Omega}v_{\rho}>0.\end{equation}
Assuming the contrary, we see that there exists a sequence $\{\rho_i\}_{i=1}^\infty$ such that $\lim_{i\to\infty}\rho_{i}=0$ and
$\lim_{i\to\infty}\inf_{x\in\Omega} v_{\rho_{i}}=0$.
Since $$\|h(v_{\rho})+q(\rho w_\rho)w_\rho\|_{\infty}\le \max_{v\in{[0,\overline C]}}|h(v)|+\max_{u\in[0,a]}q(u)\overline C,$$
it follows from Lemma \ref{harnack} that $\lim_{i\to\infty}\sup_{x\in\Omega} v_{\rho_{i}}=0$, and consequently, $$\lim_{i\to\infty}\sup_{x\in\Omega} \rho_{i}v_{\rho_{i}}=0.$$
Therefore,
$$\int_{x\in\Omega}w_{\rho_{i}}q(\rho_i w_{\rho_{i}})\left[f(\rho_i w_{\rho_{i}})-v_{\rho_{i}}\right]dx>0$$
for sufficiently large $i$,
which is a contradiction.
Letting $M=\min\{M_1,M_2\}$, we have
$$\inf_{0\le \rho\le M}\inf_{x\in \Omega}w_{\rho}>0\;\;\text{and}\;\; \inf_{0\le \rho\le M}\inf_{x\in\Omega}v_{\rho}>0.$$
This completes the
proof.
\end{proof}
From Theorem \ref{imesti}, we obtain the non-existence of non-constant positive solutions of system \eqref{steadywz1.4} for small $\rho$.
\begin{theorem}\label{mainr}
Assume that $f$, $g$ and $h$ satisfy assumptions $(\mathbf{A'_1})$, $(\mathbf{A'_2})$ and $(\mathbf{A_5})$, and $\Omega$ is a bounded domain in $\mathbb{R}^N(N\le3)$ with a smooth boundary $\partial\Omega$. Then the following two statements are true.
\begin{enumerate}
\item [$(1)$] If $f(0)>v_0$ and $h$ satisfies $(\mathbf{A_4})$, where $v_0$ is defined as in Eq. \eqref{pu2}, then there exists $c_0>0$, depending on $f$, $g$, $h$, $d_1$, $d_2$ and $\Omega$, such that system \eqref{1.4} has a unique constant positive steady state and no non-constant positive steady states for any $c>c_0$.
\item [$(2)$] If $f(0)<v_0$ and $h(v)>0$ for $v<v_0$, then there exists $c_0>0$, depending on $f$, $g$, $h$, $d_1$, $d_2$ and $\Omega$, such that system \eqref{1.4} has no positive steady states for any $c>c_0$.
\end{enumerate}
\end{theorem}
\begin{proof}
Since $f(0)>v_0$, it follows from Eq. \eqref{pu2} that $h(f(0))<0$. This relation and Lemma \ref{rho0} imply that system \eqref{steadywz1.4} has a unique positive solution $(w_*,v_*)=\left(-\f{h(f(0))}{g'(0)},f(0)\right)$ for $\rho=0$.
 Since $h$ satisfies $(\mathbf{A_4})$, we have $h'(f(0))\le0$, which implies that $(w_*,v_*)$ is non-degenerate in the sense that zero is not the eigenvalue of the linearized problem with respect to $(w_*,v_*)$. Then it follows from the implicit function theorem that there exists $\rho_0>0$ such that system \eqref{steadywz1.4} has a constant positive solution $(w_\rho,v_\rho)$ for $0<\rho<\rho_0$. Therefore, the existence is proved, and the uniqueness is proved in the following. By the implicit function theorem, we only need to show that if $(w^\rho,v^\rho)$ is a positive solution of system \eqref{steadywz1.4}, then
$$(w^\rho,v^\rho)\to (w_*,v_*)\;\;\text{in}\;\;C^{2}( \overline \Omega)\times C^{2}( \overline \Omega)\text{ as }\rho\to0.$$
It follows from Theorem \ref{imesti} that
\begin{equation}\label{cstar}
\underline C\le w^\rho,v^\rho\le\overline C \;\;\text{for}\;\;x\in\overline \Omega \text{ and } \rho\in[0,M],
\end{equation}
where $M$ is defined as in Theorem \ref{imesti} and \begin{equation*}\underline C=\min\{\inf_{0\le \rho\le M}\inf_{x\in \Omega}w^{\rho}, \inf_{0\le \rho\le M}\inf_{x\in \Omega}v^{\rho}\}>0.\end{equation*}
This, combined with the $L^p$ theory, implies that $w^\rho$ and $v^\rho$ are bounded in $W^{2,p}(\Omega)$ for any $p>N$. It follows from the embedding theorem that $w^\rho$ and $v^\rho$ are precompact in $C^1(\overline\Omega)$.
This implies that, for any sequence $\{\rho_i\}_{i=1}^\infty$ satisfying $\lim_{i\to\infty}\rho_i=0$, there exist a subsequence $\{\rho_{i_k}\}_{k=1}^\infty$  and $(w^*(x),v^*(x))\in C^1(\overline \Omega)\times C^1(\overline \Omega)$ such that
$$(w^{\rho_{i_k}},v^{\rho_{i_k}})\to (w^*(x),v^*(x))\;\;\text{in}\;\;C^1(\overline \Omega)\times C^1(\overline \Omega)\text{ as }k\to\infty,$$
where $w^*(x)$ and $v^*(x)$ is positive from Eq. \eqref{cstar}.
Note that
\begin{equation}\label{spp}
\begin{split}
w^{\rho_{i_k}}= &[-d_1\Delta+I]^{-1}\left\{ w^{\rho_{i_k}}+w^{\rho_{i_k}}q(\rho_{i_k}w^{\rho_{i_k}})\left[f(\rho_{i_k}w^{\rho_{i_k}})-v^{\rho_{i_k}}\right]\right\},\\
v^{\rho_{i_k}}=&[-d_2\Delta+I]^{-1}\left\{ v^{\rho_{i_k}}+v^{\rho_{i_k}}\left[f(v^{\rho_{i_k}})+q(\rho_{i_k}w^{\rho_{i_k}})w^{\rho_{i_k}}\right]\right\},
\end{split}
\end{equation}
and  $\lim_{k\to\infty}\rho_{i_k}w^{\rho_{i_k}}=0$ in $C^1(\overline \Omega)$. Then, taking the limit of Eq. \eqref{spp}
as $k\to\infty$ and by the Schauder theorem, we see that
$(w^*(x),v^*(x))$ is a positive solution of system \eqref{steadywz1.4} for $\rho=0$, and
$$(w^{\rho_{i_k}},v^{\rho_{i_k}})\to (w^*(x),v^*(x))\;\;\text{in}\;\;C^{2}(\overline\Omega)\times C^{2}(\overline\Omega)\text{ as }k\to\infty.$$
This completes the proof of part $(1)$.

Then, we prove the part $(2)$. It follows from $f(0)<v_0$ and Eq. \eqref{pu2} that $v_0>0$ and $v_0=\max\{v>0:h(v)=0\}$.
Then \begin{equation}\label{inev}
h(v)>0\;\;\text{for}\;\; v<v_0,\;\;\text{and}\;\;h(v)<0\;\;\text{for}\;\; v>v_0.
\end{equation} Assuming the contrary, we see that there exists
a sequence $\{\rho_j\}_{j=1}^\infty$ such that $\lim_{j\to\infty} \rho_j=0$ and system \eqref{steadywz1.4} has a positive solution
$(w_{\rho_j},v_{\rho_j})$ for $\rho=\rho_j$.
Since \begin{equation*}
-d_2\Delta v_{\rho_j}=v_{\rho_j}\left[h(v_{\rho_j}) +q({\rho_j} w_{\rho_j})w_{\rho_j}\right]\ge v_{\rho_j} h(v_{\rho_j}),
\end{equation*}
it follows from \eqref{inev} that $v_{\rho_j}\ge v_0$, which implies that
$$-d_1\Delta w_{\rho_j}\le w_{\rho_j}q(\rho_jw_{\rho_j})\left[f(\rho_jw_{\rho_j})-v_0\right].$$
From Eq. \eqref{esup}, we see that $\lim_{j\to\infty}\rho_{j}w_{\rho_{j}}=0$ uniformly on $\overline\Omega$. This, together with the fact that $f(0)<v_0$, imply that there exists $j_0>0$ such that $f(\rho_j w_{\rho_j})-v_0<0$ for any $j>j_0$. Hence  $w_{\rho_j}\le0$ for $j>j_0$. This contradicts with the fact that $w_{\rho_j}$ is positive.
\end{proof}
\begin{remark}\label{rema}
We remark that assumptions $(\mathbf{A_4})$ and $(\mathbf{A_5})$ are not strong.
Some examples of $h$ satisfying $(\mathbf{A_5})$ are
\begin{enumerate}
\item [(1)]logistic:\begin{equation} \label{log}h(v)=\beta (d-v),\;\;\beta,d>0,\end{equation}
\item [(2)] weak Allee effect: \begin{equation}\label{wall}h(v)=\beta (d-v)(v+p),\;\;d>p>0,\beta>0,\end{equation}
\item [(3)] strong Allee effect: \begin{equation}\label{sall} h(v)=\beta (d-v)(v-p),\;\;d,p,\beta>0\end{equation}
\item [(4)]strong Allee effect: \begin{equation}\label{sall2} h(v)=\beta \ds\f{(d-v)(v-p)}{v+r},\;\;d,p,\beta,r>0.\end{equation}
\end{enumerate}
Moreover, Eq. \eqref{log} always satisfies $(\mathbf{A_4})$, Eq. \eqref{wall} satisfies $(\mathbf{A_4})$ for $f(0)>d-p$, Eq. \eqref{sall} satisfies $(\mathbf{A_4})$ for $f(0)>d+p$, and Eq. \eqref{sall2} satisfies $(\mathbf{A_4})$ for $$f(0)>\ds\f{dp+dr+pr}{r}.$$
\end{remark}
\section{Applications}
In this section, we apply the previously obtained results to some concrete predator-prey models.
\subsection{A predator-prey model with Holling-II functional response}
In this subsection, we consider model \eqref{DL}, which is a predator-prey model with logistic growth rate for predator and Holling type-II functional response. Here $\Omega$ is a bounded domain in $\mathbb{R}^N(N\le3)$ with a smooth boundary $\partial\Omega$, parameters $a$, $b$, $e$, $m$, $d_1$ and $d_2$ are all positive constants, and $d$ may be positive constant, negative constant or zero.
Letting
\begin{equation}\label{H1}
  f(u)=\ds\f{(1+mu)(a-u)}{b},\;\;g(u)=\ds\f{bu}{1+mu},\;\;h(v)=d-v,\;\;\text{and}\;\; c=\ds\f{e}{b},
\end{equation}
system \eqref{1.4} is transformed to system \eqref{DL}, and parameter $c$ in system \eqref{1.4} is equivalent to $e$ in system \eqref{DL}. In this case, $f$, $g$ and $h$ satisfy assumptions $(\mathbf{A'_1})$, $(\mathbf{A'_2})$, $(\mathbf{A_4})$ and $(\mathbf{A_5})$, and
 \begin{equation}\label{pu22}
v_0=\begin{cases}
d, & \text{ if } d\ge0,\\
  0, & d<0,
\end{cases}
\end{equation}
where $v_0$ is defined as in Eq. \eqref{pu2}. Moreover, $h$ satisfies $(\mathbf{A_3})$ if $d>0$.
Then, from Theorem \ref{mainr}, we have the following results on the non-existence of non-constant steady states when conversion rate $e$ is large. This results supplements the result in \cite{Peng-Shi}, which consider the case that $e$ equals to $b$ and is sufficiently large.
\begin{proposition}\label{Pen}
\begin{enumerate}
\item [$(a_1)$] If $a>bd$, then there exists $e_0>0$, depending on $a$, $b$, $d$, $m$, $d_1$, $d_2$ and $\Omega$, such that system \eqref{DL} has a unique constant positive steady state and no non-constant positive steady states for any $e>e_0$.
\item [$(a_2)$] If $a<bd$, then there exists $e_0>0$, depending on $a$, $b$, $d$, $m$, $d_1$, $d_2$ and $\Omega$, such that system \eqref{DL} has no positive steady states for any $e>e_0$.
\end{enumerate}
\end{proposition}
From Corollary \ref{cc25}, we have the following results on the global attractivity of constant equilibria for small conversion rate, which also imply the non-existence of non-constant steady states.
\begin{proposition}
 \begin{enumerate}
\item [$(b_1)$] If $d<0$, then there exists $e_0>0$, depending on $a$, $b$, $m$ and $d$, such that the steady state $(a,0)$ of system \eqref{DL} is globally attractive for $e\in(0,e_0)$.
\item [$(b_2)$] If $0<d<\ds\f{a}{b}$, then there exists $e_0>0$, depending on $a$, $b$, $m$ and $d$, such that, for $e\in(0,e_0)$, system \eqref{DL} has a unique constant positive steady state, which is globally attractive.
\item [$(b_3)$] If $d>\f{(am+1)^2}{4mb}$, then the steady state $(0,d)$ of system \eqref{DL} is globally attractive.
\end{enumerate}
\end{proposition}
Similarly, using upper and lower solution method and modifying the arguments in Theorem \ref{T2.2}, we can prove the global attractivity of the constant positive equilibrium with respect to other parameters, for example, $a$ and $m$.
\begin{proposition}\label{p43}
If $d>0$, then there exists $a_0>0$, depending on $b$, $d$, $e$ and $m$, such that, for any $a>a_0$, system \eqref{DL} has a unique constant positive steady state, which is globally attractive.
\end{proposition}
\begin{proof}
Letting $a_1:=b\left(d+\ds\f{e}{m}\right)$, we have
$$h(f(0))=d-\ds\f{a}{b}<-cg(a)=-\ds\f{ea}{1+ma}$$for any $a>a_1$. This, together with $h(0)=d>0$, imply that the assumption of Lemma \ref{L2.1} hold. Then there exist $(\underline u,\underline v),\;(\overline
u,\overline v)>(0,0)$, which is a pair of coupled upper and lower solution of system \eqref{DL} for $a>a_1$.
Using the similar arguments as Theorem \ref{T2.2}, there exist $(\tilde u,\tilde v)$ and
$(\check u, \check v)$ such that $(\underline u,\underline v)\le(\check u, \check v)\le (\tilde u,\tilde v)\le(\overline
u,\overline v)$, $\lim_{m
\rightarrow \infty}\overline u^{(m)}=\tilde u$, $\lim_{m \rightarrow
\infty}\overline v^{(m)}=\tilde v$, $\lim_{m \rightarrow
\infty}\underline u^{(m)}=\check u$, $\lim_{m \rightarrow
\infty}\underline v^{(m)}=\check v$, and Eqs. \eqref{2.5} and \eqref{fgqu} hold.
If $$a>a_2:=\max\left\{b\left(d+\ds\f{e}{m}\right),\ds\f{1}{m}\right\}, $$ then  $$f(\check u)=\tilde v<\left(d+\ds\f{e}{m}\right)<\ds\f{a}{b}=f(0),$$ which
implies that $\tilde u>\check u>\tilde \la=\ds\f{am-1}{m}>0$, where $f(\tilde \la)=0$. It follows from Eq. \eqref{fgqu} that
\begin{equation}
\begin{split}
0=&h(f(\tilde u))-c g(\tilde u)-h(f(\check u))+c g(\check u)\\
=&\ds\f{(1+m\check u)(a-\check u)}{b}-\ds\f{(1+m\tilde u)(a-\tilde u)}{b}+\ds\f{e\check u}{1+m\check u}-\ds\f{e\tilde u}{1+m\tilde u}\\
\ge&\left[ \ds\f{1}{b}\left(2m\tilde \la-am+1\right)-e\right](\tilde u-\check u)\\
\ge& \left[ \ds\f{1}{b}\left(ma-1\right)-e\right](\tilde u-\check u),
\end{split}
\end{equation}
which implies that $\tilde u=\check u$ for $a>a_3:=\ds\f{be+1}{m}$.
Therefore, for any $a>a_0=\max\{a_1,a_2,a_3 \}$, system \eqref{DL} has a unique constant positive steady state, which is globally attractive.
\end{proof}
Similarly, we have the global attractivity of the positive equilibrium with respect to saturation $m$.
 \begin{proposition}\label{p44}
 If If $0<d<\ds\f{a}{b}$, then there exists $m_0>0$, depending on $a$, $b$, $d$ and $e$, such that, for any $m>m_0$, system \eqref{DL} has a unique constant positive steady state, which is globally attractive.
\end{proposition}
Propositions \ref{p43} and \ref{p44} supplement Theorem 2.3 in \cite{Du-Lou}, which prove the global attractivity of the positive equilibrium for $ma\le1$.
\subsection{A predator-prey model with weak Allee effect in predator and Holling-II functional response}
In this subsection, we consider the following predator-prey model
\begin{equation}\label{4weaka1}
\begin{cases}
 \ds\frac{\partial u}{\partial t}-d_1\Delta u=  u(a-u)-\ds\frac{buv}{1+mu}, & x\in \Omega,\; t>0,\\
 \ds\frac{\partial v}{\partial t}-d_2\Delta v= \beta v(d-v)(v+p)+\ds\frac{euv}{1+mu}, & x\in\Omega,\; t>0,\\
    \partial_\nu u=\partial_\nu v=0,& x\in \partial \Omega,\;
 t>0,\\
u(x,0)=u_0(x)\ge(\not\equiv)0, \;\; v(x,0)=v_0(x)\ge(\not\equiv)0,& x\in\Omega.
\end{cases}
\end{equation}
where $a$, $b$, $d$, $p$, $e$, $m$, $\beta$, $d_1$ and $d_2$ are positive constants, $d>p$, and $\Omega$ is a bounded domain in $\mathbb{R}^N(N\le3)$ with a smooth boundary $\partial\Omega$.
Here $d>p$ means that the growth rate of predator is weak Allee type in the absence of prey.
Letting
\begin{equation}\label{4H1}
  f(u)=\ds\f{(1+mu)(a-u)}{b},\;\;g(u)=\ds\f{bu}{1+mu},\;\;h(v)=\beta(d-v)(v+p),\;\;\text{and}\;\; c=\ds\f{e}{b},
\end{equation}
system \eqref{1.4} is transformed to system \eqref{4weaka1}, and parameter $c$ in system \eqref{1.4} is equivalent to $e$ in system \eqref{4weaka1}. In this case, $f$, $g$ and $h$ satisfy $(\mathbf{A_1})$, $(\mathbf{A_2})$,  $(\mathbf{A'_2})$, $(\mathbf{A_3})$ and $(\mathbf{A_5})$, $v_0=d$, and $h(v)>0$ for $v<v_0$. It follows from Remak \ref{rema} that $h(v)$ satisfies $(\mathbf{A_4})$ if $\ds\f{a}{b}=f(0)>d-p$. Then, from Theorem \ref{mainr} we have the following results on the non-existence of non-constant steady states when conversion rate $e$ is large.
\begin{proposition}
\begin{enumerate}
\item [$(a_1)$] If $a>bd$, then there exists $e_0>0$, depending on $a$, $b$, $d$, $p$, $\beta$, $m$, $d_1$, $d_2$ and $\Omega$, such that system \eqref{4weaka1} has a unique constant positive steady state and no non-constant positive steady states for $e>e_0$.
\item [$(a_2)$] If $a<bd$, then there exists $e_0>0$, depending on $a$, $b$, $d$, $p$, $\beta$, $m$, $d_1$, $d_2$ and $\Omega$, such that system \eqref{4weaka1} has no positive steady states for $e>e_0$.
\end{enumerate}
\end{proposition}
From Corollary \ref{cc25}, we have the following results on the global attractivity of constant equilibria for small conversion rate, which also imply the non-existence of non-constant steady states.
\begin{proposition}\label{wpp}
 \begin{enumerate}
\item [$(b_2)$] If $d<\ds\f{a}{b}$, then there exists $e_0>0$, depending on  $a$, $b$, $d$, $p$, $\beta$ and $m$, such that, for $e\in(0,e_0)$, system \eqref{4weaka1} has a unique constant positive steady state, which is globally attractive.
\item [$(b_3)$] If $d>\f{(am+1)^2}{4mb}$, then the steady state $(0,d)$ of system \eqref{4weaka1} is globally attractive.
\end{enumerate}
\end{proposition}

\subsection{A predator-prey model with strong Allee effect in predator and Holling-IV functional response}
In this subsection, we consider the following model,
\begin{equation}\label{4weaka}
\begin{cases}
  \ds\frac{\partial u}{\partial t}-d_1\Delta u= u(a-u)-\ds\frac{buv}{1+nu+mu^2}, & x\in \Omega,\; t>0,\\
 \ds\frac{\partial v}{\partial t}-d_2\Delta v=\beta v(d-v)(v-p)+\ds\frac{euv}{1+nu+mu^2}, & x\in\Omega,\; t>0,\\
    \partial_\nu u=\partial_\nu v=0,& x\in \partial \Omega,\;
 t>0,\\
\end{cases}
\end{equation}
which is a diffusive predator-prey model with strong Allee effect in predator. Here $a$, $b$, $d$, $p$, $e$, $m$, $n$, $\beta$, $d_1$ and $d_2$ are positive constants.
Letting
\begin{equation}\label{4H2}
\begin{split}
  &f(u)=\ds\f{(1+nu+mu^2)(a-u)}{b},\;\;c=\ds\f{e}{b},\\
  &g(u)=\ds\f{bu}{1+nu+mu^2},\;\;h(v)=\beta (d-v)(v-p).
  \end{split}
\end{equation}
Then system \eqref{1.4} is transformed to system \eqref{4weaka}, and parameter $c$ in system \eqref{1.4} is equivalent to $e$ in system \eqref{4weaka}.
In this case, $f$, $g$ and $h$ satisfy $(\mathbf{A'_1})$, $(\mathbf{A'_2})$ and $(\mathbf{A_5})$, and $v_0=\max\{d,p\}$. It follows from Remak \ref{rema} that $h(v)$ satisfies assumptions $(\mathbf{A_4})$ if $\ds\f{a}{b}=f(0)>d+p$.  Then from Theorem \ref{mainr}, we have the following results on the non-existence of non-constant steady states when conversion rate $e$ is large.
\begin{proposition}
If $a>b(d+p)$, then there exists $e_0>0$, depending on $a$, $b$, $d$, $p$, $m$, $\beta$, $d_1$, $d_2$ and $\Omega$, such that system \eqref{4weaka} has a unique constant positive steady state and no non-constant positive steady states for any $e>e_0$.
\end{proposition}
\section{Conclusions}
In this paper, we consider a general diffusive predator-prey system. It covers a wide range of predator-prey models which include some well-known ones but also some less studied ones. We find that the conversion rate is a key parameter to affect the dynamics of a general predator-prey model, and there are almost no complex patterns for large and small conversion rate. Hence, this phenomenon can occur commonly for predator-prey models, which was found in \cite{Cheny} for a special model with a nonlinear growth rate for the predator.

For the case of small conversion rate, we show the global attractivity of the unique constant positive steady state even when
$h(v)$ is nonmonotonic, and hence it can be applied to  model \eqref{4weaka1} with weak Allee effect in predator. A special case where $h(v)$ is monotonic was analyzed in \cite{Cheny}. We remark that our result is not a direct conclusion from \cite{27,pao,4,5}, and needs some detailed analysis for the relations between the limits of upper solutions sequence and lower solutions sequence. Moreover, our result supplements some existing ones. For example, Propositions \ref{p43} and \ref{p44} supplement Theorem 2.3 in \cite{Du-Lou}, which show the global attractivity of the positive equilibrium for $ma\le1$.

For the case of large conversion rate, we show the nonexistence of the positive steady states. Our method is motivated by \cite{Peng-Shi}, but we need to modify many of their arguments to derive our result. We find that, even with a nonmonotonic functional response, there exist no nonconstant positive steady states for large conversion rate.
Moreover, our result in Theorem \ref{Pen} also supplements Theorem 1.2 of \cite{Peng-Shi}, which show the nonexistence in the case that $e$ equals to $b$ and is sufficiently large.

Finally, we should mention that the dimension $N$ of the domain $\Omega$ we studied is less than three, and this is meaningful in biology.
For a general dimension $N$, we can also obtain the similar results, if $h$ satisfies the following assumption:
\begin{enumerate}
\item [$(\mathbf{A'_5})$]  $h\in C^1(\overline{\mathbb {R}^+})$ and there exist $n\in \mathbb{N}^+$, $\{q_i\}_{i=0}^{n}$, $\{k_i\}_{i=0}^n$ and $\{\overline k_i\}_{i=0}^n$ such that $$\sum_{i=0}^n k_i v^{q_i}\le -h(v)\le \sum_{i=0}^n \overline k_i v^{q_i}\;\;\text{for any}\;\;v\ge0,$$
    where $0= q_0<q_1<q_2<\dots< q_n$, $q_n>\ds\f{N}{2}-1$ and $k_n,\overline k_n>0$.
\end{enumerate}

\end{document}